\documentclass[a4paper,american,reqno]{amsart}

\newif\ifPreprint \Preprinttrue
\newif\ifSubmission \Submissionfalse

\usepackage{babel}
\usepackage[utf8]{inputenc}
\usepackage[T1]{fontenc}
\usepackage{xspace}
\usepackage{todonotes}
\usepackage{csquotes}
\usepackage{lmodern}
\usepackage{microtype}
\usepackage[style = numeric-comp,
            firstinits = true,
            maxbibnames = 100,
            maxcitenames = 2,
            isbn = false,
            backend = biber]{biblatex}
\usepackage[colorlinks,
            citecolor=blue,
            urlcolor=blue,
            linkcolor=blue]{hyperref}
\usepackage{cleveref}

\makeatletter
\patchcmd{\@settitle}{\uppercasenonmath\@title}{\scshape\large}{}{}
\patchcmd{\@setauthors}{\MakeUppercase}{\scshape\normalsize}{}{}
\makeatother

\vfuzz \hfuzz
\newcommand{\rev}[1]{#1}

\theoremstyle{plain}

\newtheorem{corollary}{Corollary}
\newtheorem{theorem}{Theorem}
\theoremstyle{definition}

\newtheorem{example}{Example}

\newtheorem{remark}{Remark}

\theoremstyle{remark}

\newcommand{\set}[1]{\{#1\}}
\newcommand{\Set}[1]{\left\{#1\right\}}
\newcommand{\defset}[3][\defsep]{\set{#2#1#3}}
\newcommand{\Defset}[3][\defsep]{\Set{#2#1#3}}
\newcommand{\Unc}{\mathcal{U}}

\newcommand{\st}{\text{s.t.}}

\newcommand{\R}{\mathbb{R}}

\newcommand{\N}{\mathbb{N}}
\renewcommand{\P}{\mathcal{P}}

\newcommand{\fa}{\text{ for all }}

\newcommand{\abbr}[1][abbrev]{#1.\xspace}

\newcommand{\ie}{\abbr[i.e]}

\newcommand{\define}{\mathrel{{\mathop:}{=}}}

\newcommand{\relint}{\text{relint}}

\newcommand{\lin}{\text{lin}}
\newcommand{\bdt}{\boldsymbol{\cdot}}

\bibliography{aar-lcp-follow-up}

\begin{document}

\title[On AAR-LCPs and MILPs]%
{On the Relation Between Affinely Adjustable Robust Linear
  Complementarity and Mixed-Integer Linear Feasibility Problems}
\author[C. Biefel, M. Schmidt]%
{Christian Biefel, Martin Schmidt}

\address[C. Biefel]{}
\email{cm.biefel@gmail.com}

\address[M. Schmidt]{%
  Trier University,
  Department of Mathematics,
  Universitätsring 15,
  54296 Trier,
  Germany}
\email{martin.schmidt@uni-trier.de}

\date{\today}

\begin{abstract}
  We consider adjustable robust linear complementarity problems
and extend the results of Biefel et al.~(2022) towards convex and
compact uncertainty sets.
Moreover, for the case of polyhedral uncertainty sets, we prove that
computing an adjustable robust solution of a given linear
complementarity problem is equivalent to solving a properly chosen
mixed-integer linear feasibility problem.

%%% Local Variables:
%%% mode: latex
%%% TeX-master: "aar-lcp-follow-up-preprint"
%%% End:

\end{abstract}

\keywords{Linear Complementarity Problems,
Adjustable Robustness,
Robust Optimization,
Mixed-Integer Linear Optimization%
%
%
%%% Local Variables:
%%% mode: plain-tex
%%% TeX-master: "aar-lcp-follow-up-preprint"
%%% End:
}
\subjclass[2010]{90C33, % Complementarity and equilibrium problems and variational inequalities (finite dimensions)
91B50, % General equilibrium theory
90Cxx, % Mathematical programming
90C34% Semi-infinite programming
%
%
%%% Local Variables:
%%% mode: plain-tex
%%% TeX-master: "aar-lcp-follow-up-preprint"
%%% End:
}

\maketitle

\section{Introduction}
\label{sec:introduction}

We consider affinely adjustable robust (AAR) linear
complementarity problems (LCPs).
The classic, \ie, deterministic, LCP is defined as follows.
Given a matrix $M \in \R^{n\times n}$ and a vector $q \in \R^n$, the
LCP($q, M$) is the problem to find a vector $z \in \R^n$ that
satisfies the conditions
\begin{equation}
  \label{eq:LCPold}
  z\geq 0,
  \quad Mz+q\geq 0,
  \quad z^\top(Mz+q)=0
\end{equation}
or to show that no such vector exists.
In the following, we use the standard $\perp$-notation and
abbreviate~\eqref{eq:LCPold} as
\begin{equation}
  \label{eq:LCP}
  0 \leq z \perp Mz + q \geq 0.
\end{equation}

LCPs are very important both in applications as well as in
mathematical theory itself.
For instance, they are used to model market equilibrium problems in
many applied studies of gas or electricity markets
\cite{Gabriel_et_al:2012} but also play an important role in
mathematical optimization, game theory, or general matrix theory.
We refer the interested reader to the seminal book
\cite{Cottle_et_al:2009} for an overview.

Although there is a very strong connection between LCPs and mathematical
optimization and although the latter has been studied a lot in the
recent decades under data uncertainty, the field of LCPs under
uncertainty is still in its infancy.
Stochastic approaches can be found in
\cite{Chen_Fukushima:2005,Chen_et_al:2012,Chen_et_al:2009,Lin_Fukushima:2006}
and are mainly based on minimizing the expected residual gap of the
uncertain LCP.
On the other hand, robust approaches for uncertain LCPs have been
considered recently as well.
The first rigorous analysis of robust LCPs can be found in
\cite{Xie_Shanbhag:2014,Xie_Shanbhag:2016}, where the authors apply
the concept of strict robustness \cite{Soyster:1973} to LCPs, which has
been used later in~\cite{Mather_Munsing:2017} in the context of
Cournot--Bertrand equilibria in power networks.
Moreover, in \cite{Krebs_Schmidt:2020,Krebs_et_al:2022}, LCPs have
been studied using $\Gamma$-robustness as introduced in
\cite{Bertsimas_Sim:2003,Sim:2004,Bertsimas_Sim:2004}; see
\cite{Celebi_et_al:2020,Kramer_et_al:2021} for some applications in
power markets.

The most recent paper on robust LCPs, to the best of our knowledge, is
\cite{Biefel_et_al:2022}, where robust LCPs are studied using the
concept of adjustable robustness
\cite{Ben-Tal_et_al:2004,Yanikoglu_et_al:2019}.
In \cite{Biefel_et_al:2022}, the authors study adjustable robust
LCPs in the most simplest setting, which is for affine decision
rules and box uncertainties.
In this short note, we stay with affine decision rules but generalize
the results to general convex and compact uncertainty
sets~$\mathcal{U}$.
In this context, our contribution is twofold.
First, we characterize AAR solutions of robust LCPs and, second, use
this characterization to prove that the AAR LCP with a polyhedral
uncertainty set is equivalent to a properly chosen mixed-integer
linear problem (MILP).

Let us finally note that our study is related to
\cite{Adelgren_Wiecek:2016}, where the authors consider
multi-parametric LCPs for sufficient matrices~$M$.
However, our robust approach as well as the studied relation to MILPs
differ from the concepts and results of \cite{Adelgren_Wiecek:2016}.

We introduce the problem under consideration in
Section~\ref{sec:problem-statement} and derive our main results in
Section~\ref{sec:main-results}.
Afterward, we comment on some special cases and extensions in
Section~\ref{sec:remarks}.

%%% Local Variables:
%%% mode: latex
%%% TeX-master: "aar-lcp-follow-up-preprint"
%%% End:

\section{Problem Statement}
\label{sec:problem-statement}

We now define the adjustable robust LCP with affine decision rules.
To this end, let $M \in \R^{n \times n}$ and $q \in \R^n$ as before
and let $T \in \R^{n \times k}$ be given.
We assume that $q$ is perturbed by $Tu$ with $u\in\Unc$.
In what follows, we assume that $\Unc \subset \R^k$ is a convex and
compact uncertainty set \rev{that, w.l.o.g., contains $0$ in its relative interior, \ie,}
$0 \in \relint(\Unc)$.
Then, the affinely adjustable robust LCP($q, M, T, \Unc$) consists of
finding an affine decision rule, \ie, we want to determine $D \in
\R^{n \times k}$ and $r \in \R^n$ such that $z(u) = Du + r$ satisfies
\begin{equation}
  0 \leq z(u) \perp Mz(u) + q(u) \geq 0 \quad \fa u \in \Unc.
\end{equation}
Equivalently, we can state the problem more explicitly as
\begin{equation}
  \label{eq:ULCP-q-new}
  0 \leq Du + r \perp MDu + Mr + q + Tu \geq 0 \quad \fa u \in \Unc.
\end{equation}
Without loss of generality, we may assume that $T \in \R^{n \times k}$
has full column rank; see~\cite{Adelgren_Wiecek:2016}.
\rev{In many applications, some variables are non-adjustable and thus
  have to be fixed before the uncertainty realizes.}
To model \rev{these so-called} here-and-now variables,
we simply require that the first $h$ rows of $D$ are zero \rev{for some $h<n$}.
For more details, we refer to \cite{Biefel_et_al:2022}.

We close this section by brief\/ly introducing the following
notation.
Let $A \in \R^{m\times n}$, $b \in \R^m$, and index
sets~$I\subseteq[m]\define \set{1, \dotsc, m}$ as well as $J
\subseteq [n]$ be given.
Then, $A_{I,J} \in \R^{|I| \times |J|}$ denotes the submatrix of~$A$
consisting of the rows indexed by~$I$ and the columns indexed by~$J$.
Moreover, $b_I$ denotes the subvector with components specified by
entries in~$I$. If $I = J$, we also write $A_I$ instead of $A_{I,I}$.

%%% Local Variables:
%%% mode: latex
%%% TeX-master: "aar-lcp-follow-up-preprint"
%%% End:

\section{Main Results}
\label{sec:main-results}

In this section, we state and prove our two main results.
The first one is a full characterization of AAR solutions of robust LCPs.

\begin{theorem}
  \label{thm:complementarity}
  Assume that $\Unc$ is convex and compact with $0\in\relint(\Unc)$
  and let $\mathcal{B} = \set{v^1, \dotsc, v^\ell}$, $\ell \in \N$
  \rev{with $\ell \leq k$}, be
  a basis of the linear hull~$\lin(\Unc)$ of $\Unc$.
  Moreover, let $z(u) = Du + r$ such that $z(u) \geq 0$ as well as
  $Mz(u) + q + Tu\geq 0$ holds for all $u \in \Unc$ and define
  $I \define \defset{i \in [n]}{r_i > 0}$.
  Then, $z(u)=Du+r$ is an AAR solution if and only if $D$ and $r$ satisfy the
  conditions
  \begin{align}
    M_{I,\bdt}r + q_I & = 0, \label{eq:compl-dnom}\\
    (M_{I,\bdt}D + T_{I,\bdt})v^j & = 0, \quad j \in [\ell].\label{eq:compl-linhull}
  \end{align}
\end{theorem}
\begin{proof}
  First, let $z(u) = Du + r$ be an AAR solution.
  Then, $r$ is a nominal solution (as $0 \in \Unc$) and therefore $r$
  satisfies $M_{I,\bdt}r + q_I = 0$, \ie, \eqref{eq:compl-dnom} is
  fulfilled.
  For every $v^j$, $j \in \rev{[\ell]}$, there exists a scalar~$\delta_j > 0$ such
  that $\delta_jv^j \in \Unc$ and $\delta_j D_{I,\bdt} v^j + r_I > 0$ holds.
  Thus, for every $j \in [\ell]$ the AAR solution $z$ satisfies
  \begin{equation*}
    0 = M_{I,\bdt} z(\delta_jv^j) + q_I + \delta_jT_{I,\bdt}v^j
    = \delta_jM_{I,\bdt}Dv^j + \delta_jT_{I,\bdt}v^j
    = \delta_j(M_{I,\bdt}D+T_{I,\bdt})v^j,
  \end{equation*}
  where we used \eqref{eq:compl-dnom} for the second equality.
  Thus, $z$ satisfies \eqref{eq:compl-linhull}.

  Let now $D$ and $r$ satisfy \eqref{eq:compl-dnom} and
  \eqref{eq:compl-linhull}.
  From $0 \in \relint(\Unc)$, it follows that for all $u \in \Unc$
  there exists an  $\varepsilon > 0$ such that $-\varepsilon u \in
  \Unc$.
  Hence, nonnegativity of $z(u)=Du+r$ yields
  \begin{align*}
    \Defset{i \in [n]}{\exists u \in \Unc : D_{i,\bdt}u + r_i > 0}
    \subseteq I.
  \end{align*}
  Thus, for $\bar{I} = [n] \setminus I$,
  \begin{equation*}
    z_{\bar{I}}(u)^\top (M_{\bar{I},\bdt} z(u) + q_{\bar{I}} +
    T_{\bar{I},\bdt}u) = 0
  \end{equation*}
  holds for all $u \in \Unc$.
  On the other hand, every $u \in \Unc$ can be written as a linear
  combination $u = \sum_{i=1}^\ell \lambda_j v^j$ with
  $\lambda_j \in \R$.
  Hence,
  \begin{align*}
    M_{I,\bdt}(Du+r)+q_I+T_{I,\bdt}u
    =
    M_{I,\bdt}Du + T_{I,\bdt} u
    = (M_{I,\bdt}D+T_{I,\bdt}) \left(\sum_{j=1}^\ell \lambda_j v^j\right)
    = 0
  \end{align*}
  holds, where we used \eqref{eq:compl-dnom} for the first and
  \eqref{eq:compl-linhull} for the last equality.
  Therefore, $z(u)=Du+r$ fulfills complementarity and is an AAR solution
  due to the additional assumptions of the theorem.
\end{proof}

\rev{The last theorem states a rather abstract characterization of AAR
  solutions.
  For arbitrary convex and compact uncertainty sets, working with this
  characterization might be difficult.
  However, the characterization can be practically used in more
  specific cases, which is what we do in our second main result
  about polyhedral uncertainty sets, where we use the characterization
  of the last theorem to show that affinely adjustable robust
  solutions are the solutions of a properly chosen MILP.}

\begin{theorem}
  \label{thm:main}
  Let $\Unc = \defset{u \in \R^k}{\Theta u \geq \zeta}$ with
  $\Theta \in \R^{g \times k}$ and $\zeta \in \R^g$ and let
  $\mathcal{B} = \set{v^1, \dotsc, v^\ell}$ be a basis of
  $\lin(\Unc)$.
  Furthermore, let $\rev{b} \in \R$ be sufficiently large and consider the
  mixed-integer linear feasibility problem
  \begin{subequations}
    \label{eq:MIP}
    \begin{align}
      \text{Find} \quad
      &x\in\{0,1\}^n,~D\in\R^{n\times k},
        ~r\in\R_{\geq 0}^n,~A,C\in\R_{\geq 0}^{g\times n}
      \\
      \st \quad
      &r_i\leq bx_i,
      &i\in [n],\label{eq:MIP-d1}\\
      &b(1-x_i)\geq M_{i,\bdt}r+q_i\geq 0,
      &i\in [n],\label{eq:MIP-d2}\\
      &b(1-x_i)\geq (M_{i,\bdt}D+T_{i,\bdt})v^j\geq -b(1-x_i),
      &i \in [n],\, j\in[\ell],\label{eq:MIP-mD1}\\
      &\zeta^\top A_{\bdt,i} + r_i \geq 0,
      & i\in [n],\label{eq:MIP-A1}\\
      &\Theta^\top A_{\bdt,i} = D_{i,\bdt}^\top ,
      & i\in [n],\label{eq:MIP-A2}\\
      &\zeta^\top C_{\bdt,i} + M_{i,\bdt}r +q_i \geq 0,
      &i\in [n],\label{eq:MIP-C1}\\
      &\Theta^\top C_{\bdt,i} = (M_{i,\bdt}D+T_{i,\bdt})^\top ,
      &i\in [n],\label{eq:MIP-C2}\\
      &D_{[h],\bdt}=0.\label{eq:MIP-nonadj}
    \end{align}
  \end{subequations}
  If \eqref{eq:MIP} is feasible, it returns an AAR solution of the
  form $z(u) = Du + r$ of \eqref{eq:ULCP-q-new}.
  If it is infeasible, no AAR solution exists.
\end{theorem}
\begin{proof}
  We show that $z(u)=Du+r$ is an AAR solution if and only
  if there exist $x,A,C$ such that $x,D,r,A,C$ solve \eqref{eq:MIP}.

  \rev{We start by proving complementarity of the solutions.
  Let $z(u)=Du+r$ be an AAR solution.
  We define $I \define \defset{i \in [n]}{r_i > 0}$ and
  $x_i=1$ for all $i\in I$ and $x_i=0$ for all $i\in [n]\setminus I$.
  Then, \Cref{thm:complementarity} implies that $x,D,r$ satisfy the constraints
  \eqref{eq:MIP-d1}--\eqref{eq:MIP-mD1} for sufficiently large~$b$.
  On the other hand, if $x,D,r,A,C$ satisfy the conditions~
  \eqref{eq:MIP-d1}--\eqref{eq:MIP-mD1},
  $r_i>0$ implies $x_i=1$ and thus $D$ and $r$ fulfill
  the conditions~\eqref{eq:compl-dnom}
  and \eqref{eq:compl-linhull} of
  \Cref{thm:complementarity}.}

  \rev{It remains to consider the nonnegativity constraints of
  \eqref{eq:ULCP-q-new}.}
  First, we prove nonnegativity of the solution, \ie,
  $Du+r\geq 0$ for all $u\in\Unc$, if and only if there exists a
  matrix $A$ such that $D,r,A$ satisfy \eqref{eq:MIP-A1} and \eqref{eq:MIP-A2}.
  For all $i\in[n]$, we observe that $D_{i,\bdt}u+r_i\geq 0$ holds for
  all $u\in\Unc$ if and only if $\min_{u\in\Unc} \set{D_{i,\bdt}u+r_i}\geq 0$.
  We now employ duality and obtain that this is equivalent to the
  statement that there exists a vector $a\in\R^g_{\geq 0}$ such that
  $\zeta^\top a + r_i \geq 0$ and $\Theta^\top a = D_{i,\bdt}^\top $.
  The matrix~$A\in\R_{\geq 0}^{g\times n}$ then contains the
  vectors~$a$ as columns.

  Next, we show that $MDu+Mr+q+Tu\geq 0$ holds for all
  $u\in\Unc$ if and only if there exists a
  matrix $C$ such that $D,r,C$ satisfy \eqref{eq:MIP-C1} and \eqref{eq:MIP-C2}.
  This is analogous to the previous step and we observe that for every $i\in[n]$,
  $M_{i,\bdt}Du+M_{i,\bdt}r+q_i+T_{i,\bdt}u\geq 0$ for all
  $u\in\Unc$ is equivalent to \mbox{$\min_{u\in\Unc} \set{M_{i,\bdt}Du +
    M_{i,\bdt}r+q_i+T_{i,\bdt}u} \geq 0$}.
  Again, this holds if and only if there exists a vector
  $c\in\R^g_{\geq 0}$ such that $\zeta^\top c + \rev{M_{i,\bdt}r} +q_i \geq 0$ and
  $\Theta^\top c = (M_{i,\bdt}D+T_{i,\bdt})^\top $.
  The matrix~$C\in\R_{\geq 0}^{g\times n}$ then contains the
  vectors~$c$ as columns.

  Finally, the remaining constraint \eqref{eq:MIP-nonadj} enforces
  that the first $h$~variables are non-adjustable.
\end{proof}

\begin{remark}
  The linear hull $\lin(\Unc)$ of the uncertainty set $\Unc$
  can be computed in polynomial time if $\Unc$ is a polyhedron, i.e.,
  if $\Unc = \defset{u\in\R^k}{\Theta u \geq \zeta}$ as in \Cref{thm:main}.
  We can then maximize once in every direction $\Theta_{j,\bdt}$,
  $j\in[g]$, and check if the optimal value is larger than $\zeta_j$.
  If it is equal to $\zeta_j$, we know $\zeta_j=0$ due to
  $0\in\relint(\Unc)$ and the inequality constraint can be replaced by
  an equality constraint.
  We obtain the representation
  $\Unc = \defset{u\in\R^k}{\Phi u=0, \Theta' u \geq \zeta'}$ with
  $\Phi\in\R^{(g-f)\times k}$, $f\leq g$, and
  $\Theta'\in\R^{f\times k}$.
  The basis of $\lin(\Unc)$ is then given by the basis of ker$(\Phi)$.
\end{remark}

Let us also comment on a difference to the setting considered in
\cite{Biefel_et_al:2022}.
There, the submatrix~$M_I$ has to be invertible for an AAR solution to exist
if all entries of $q$ are uncertain, cf.~Theorem~4.5
in~\cite{Biefel_et_al:2022}.
This is not the case in our setting as the following example shows.

\begin{example}
  Consider the uncertain LCP given by
  \begin{align*}
    M =
    \begin{bmatrix}
      1&-1\\1&-1
    \end{bmatrix},
    \quad
    q(u) =
    \begin{pmatrix}
      -1\\-1
    \end{pmatrix}
    +
    \begin{pmatrix}
      u_1\\u_2
    \end{pmatrix},
    \quad
    \Unc = \Defset{(u_1, u_2)}{-2 \leq u_1 =u_2 \leq 2}.
  \end{align*}
  Then,
  \begin{equation*}
    D =
    \begin{bmatrix}
      -1&0\\0&0
    \end{bmatrix},
    \quad
    r =
    \begin{pmatrix}
      2\\1
    \end{pmatrix}
  \end{equation*}
  is an AAR solution with $I = \set{1,2}$, but the matrix~$M$ is not
  invertible.
\end{example}

Finally note that if $T$ is the identity matrix and $\Unc$ is
a box, the MILP~\eqref{eq:MIP} is equivalent to the MILP in
Theorem~4.7 in~\cite{Biefel_et_al:2022}.

%%% Local Variables:
%%% mode: latex
%%% TeX-master: "aar-lcp-follow-up-preprint"
%%% End:

\section{Remarks and Extensions}
\label{sec:remarks}

In this section, we comment on a special case, namely the one in which
$M$ is positive semidefinite, and several possible extensions.

%%%%%%%%%%%%%%%%%%%%%%%%%%%%%%%%%%%%%%%%%%%%%%%%%%%%%%%%%%%%%%%%%%%%%%%%%%%
\subsection{Positive Semidefinite $M$}

\rev{
We first consider the case that the matrix $M$ is positive
semidefinite.
In the following, we show that in this setting an AAR solution can be found
in polynomial time.
The same result was shown for box uncertainties in~\cite{Biefel_et_al:2022}
with similar arguments.
For positive semidefinite $M$, Theorem~3.1.7~(a) in~\cite{Cottle_et_al:2009}
states that
\begin{align}\label{eq:Cottle-a}
  y^\top(Mz+q) = z^\top(My+q) = 0
\end{align}
holds for any $y,z\in \text{SOL}(q,M)$, where $\text{SOL}(q,M)$
denotes the set of solutions of the LCP$(q, M)$.
Let
\begin{align*}
  \P = \Defset{i\in[n]}{\exists z\in\text{SOL}(q,M) \text{ with } z_i>0}.
\end{align*}
Due to \eqref{eq:Cottle-a}, every nominal solution ${r\in \text{SOL}(q,M)}$
satisfies ${M_{\P,\bdt}r + q_\P = 0}$.
Therefore, every AAR solution has to satisfy
\begin{align*}
  M_{\P,\bdt} Du + T_{\P,\bdt}u = 0
\end{align*}
for all $u\in\Unc$ as otherwise there would exist a $u'\in\Unc$ with
$M_{i,\bdt}(Du'+r)+\bar{q}+T_{i,\bdt}u'<0$ for some $i\in \P$.
Thus, the set $I$ in \Cref{thm:complementarity} can be replaced by $\P$ and
the MILP~\eqref{eq:MIP} can be simplified to an LP as we do not
need the binary variables anymore.}

\rev{Furthermore, Theorem~3.1.7~(c) in~\cite{Cottle_et_al:2009} states that
$\text{SOL}(q,M)$ is given by
\begin{align*}
  \text{SOL}(q,M) = \Defset{z \in \R^n_{\geq 0}}
  {q+Mz\geq 0, \ q^\top(z-\bar{z}) = 0, \
  (M+M^\top)(z-\bar{z}) = 0},
\end{align*}
where $\bar{z}\in\text{SOL}(q,M)$ is an arbitrary solution.
Such a solution $\bar{z}$ can be found by solving a single convex-quadratic
optimization problem.
With this polyhedral description of $\text{SOL}(q,M)$, $\P$ can be obtained
by solving $n$ linear programs in which $z_i$, $i \in [n]$, is maximized
over~$\text{SOL}(q,M)$ and then checking, whether the optimal value
is strictly positive.
This implies that $\P$ can be computed in polynomial time and, hence, we
can find an AAR solution in polynomial time if $M$ is positive semidefinite.
}

%%%%%%%%%%%%%%%%%%%%%%%%%%%%%%%%%%%%%%%%%%%%%%%%%%%%%%%%%%%%%%%%%%%%%%%%%%%
\subsection{Discrete Uncertainty Sets}

Next, we brief\/ly discuss discrete uncertainty sets.
In the following example, for any uncertainty realization in the
discrete set, there exists a solution whereas there does not
exist solutions for some realizations in the convex hull of the uncertainty
set.

\begin{example}
  Consider the LCP given by
  \begin{align*}
    M=
    \begin{bmatrix}
      0&0\\1&0
    \end{bmatrix},
    \quad
    q(u)=
    \begin{pmatrix}
      1+u\\u
    \end{pmatrix}
  \end{align*}
  and $\Unc = \set{\pm 1}$.
  Then, for $u=1$, $z=(0,0)$ is a solution and for $u=-1$, $z=(1,0)$
  is a solution.
  If $\Unc' = \text{conv}(\Unc)$, there is, however, no solution for
  $u= - 1/2$.
\end{example}

This example is in contrast to results for classic robust linear
optimization, where one can always replace the uncertainty set with its
convex hull.
The reason for this behavior can be explained with classic LCP theory.
In the literature, the cone of vectors $q$ for which the LCP($q,M$)
with a given matrix $M$ has a solution is usually denoted by $K(M)$, i.e.,
\begin{align*}
  K(M) = \Defset{q\in\R^n}{\text{SOL}(q,M) \neq \emptyset}.
\end{align*}
In general, $K(M)$ is not convex, and hence the convex hull of some points that
lie in $K(M)$ is not necessarily contained in $K(M)$.
However, $K(M)$ is convex if and only if $M$ is a so-called $Q_0$-matrix,
cf.~Proposition~3.2.1 in~\cite{Cottle_et_al:2009}, and we obtain the
following result.

\begin{corollary}
  Suppose that $M$ is a $Q_0$-matrix.
  Then, the uncertain LCP has a solution for
  all $u\in\text{conv}(\Unc)$ if it has a solution for all $u\in\Unc$.
\end{corollary}

%%%%%%%%%%%%%%%%%%%%%%%%%%%%%%%%%%%%%%%%%%%%%%%%%%%%%%%%%%%%%%%%%%%%%%%%%%%
\subsection{Decision-Dependent Uncertainty Sets}

The MILP~\eqref{eq:MIP} can be extended to cover simple
decision-dependent uncertainty sets.
To this end, consider the uncertainty set
\begin{align*}
  \Unc(r)=\defset{u\in\R^k}{\Theta u\geq \zeta + \Psi r},
  \quad
  \Psi\in\R^{g\times n},
\end{align*}
that depends on the \rev{chosen nominal solution}~$r$.
If the deviation caused by $\Psi r$ is not too large, in some cases
the linear hull does not change.
Hence, in these cases we only have to replace the constraints
\eqref{eq:MIP-A1} and \eqref{eq:MIP-C1} by their respective quadratic
versions that include the terms~$(\Psi r)^\top A_{\bdt, i}$ and $(\Psi
r)^\top C_{\bdt, i}$, respectively.
We leave the detailed study of such situations for future work.

\subsection{Mixed LCPs}
\label{sec:mixedlcp}

Finally, we discuss so-called mixed LCPs.
These problems consist in finding $z\in\R^n$ and $y\in\R^m$ such that
\begin{subequations}
  \label{eq:mixedlcp}
  \begin{align}
    Vz+Wy+p &= 0, \label{eq:mixedlcp-equality}\\
    Mz+Ny+q &\geq 0, \\
    z &\geq 0,\\
    z^\top (Mz+Ny+q) &= 0
  \end{align}
\end{subequations}
with $M\in\R^{n\times n}$, $N\in\R^{n\times m}$, $q\in\R^n$,
$V\in\R^{m\times n}$, $W\in\R^{m\times m}$, $p\in\R^m$.
We refer to~\cite{Cottle_et_al:2009} for some source problems.

We now brief\/ly demonstrate necessary adaptions to the
MILP~\eqref{eq:MIP} to compute an AAR solution to an
uncertain version of the mixed LCP~\eqref{eq:mixedlcp}.
As before, we assume that $q$ is affected by uncertainty in the form
of $q(u)=Tu$, $u\in\Unc$, and that $z$ is affinely adjustable,
i.e., $z(u)=Du+r$.
Several parameters might be uncertain in the case of mixed LCPs.
In the simplest case, the matrices $V$, $W$, $M$, and $N$ are certain,
$y$ is non-adjustable and only $p(u)=p+Pu$ is uncertain for some given
$P\in\R^{m\times k}$ and $u\in\Unc$.
In this case, the constraints~\eqref{eq:MIP-d2} and~\eqref{eq:MIP-C1}
have to be extended by the term~$Ny$.
Moreover, $D$, $r$, and $y$ have to satisfy the resulting slightly
adapted version of the MILP~\eqref{eq:MIP} and the additional
constraints
\begin{align*}
  Vr+Wy+p & = 0, \\
  (VD+P)v^j & = 0, \quad j \in [l].
\end{align*}
This also includes the special case in which all additional
parameters $p$, $V$, $W$, $M$, and $N$ are certain and $y$ is
non-adjustable.
In this case, the second of the above constraints reduces to $VDv^j=0$
for all $v^j\in[l]$.
In the case of adjustable~$y$, i.e., $y(u)=Eu+s$,
the MILP has to be adapted accordingly in a similar way.
Additionally, $z$ and $y$ have to satisfy the additional constraints
\begin{align*}
  Vr+Ws+p & = 0, \\
  (VD+WE+P)v^j & = 0, \quad j \in [l].
\end{align*}

Compared to the classic LCP, on the one hand we get additional
freedom by being allowed to choose more variables, while on the other
hand, there are additional constraints, some of which might be quite
restrictive.

%%% Local Variables:
%%% mode: latex
%%% TeX-master: "aar-lcp-follow-up-preprint"
%%% End:

%%% Local Variables:
%%% mode: latex
%%% TeX-master: "aar-lcp-follow-up-preprint"
%%% End:

\printbibliography

\end{document}